\documentclass[11pt]{article}
\usepackage{amsmath,amssymb,amsthm,authblk,color,enumerate,comment,centernot,enumitem,url,cite, mathrsfs,frcursive,float}

\usepackage{tikz}
\usetikzlibrary{arrows}
\usepackage{graphicx,relsize}
\usepackage{mathtools, enumitem}
\usepackage{array}

\newcommand{\defeq}{\stackrel{\textnormal{def}}{=}}

\theoremstyle{plain}

\newtheorem{theorem}{{Theorem}}

\theoremstyle{definition}

\theoremstyle{remark}

\def\Q {{\mathbb Q}}

\def\cM {{\mathcal M}}
\def\cP {{\mathcal P}}

\def\cY {{\mathcal Y}}

\newcommand{\Ht}{\textnormal{Ht}}

\newcommand{\Val}{\textnormal{Val}}

\author[1]{Jun-Yong Park}
\affil[1]{University of Sydney\\
\texttt{June.Park@sydney.edu.au}}

\author[2]{Tristan Phillips}
\affil[2]{Dartmouth College\\
\texttt{TristanPhillips72@gmail.com}}

\title{$100\%$ of elliptic curves with a marked point have positive rank}
\date{}

\begin{document}

\maketitle

%\textbf{$100\%$ of elliptic curves with a marked point have positive rank}
%\vspace{5mm}

\begin{abstract}
  As a consequence of their work on average Selmer ranks of elliptic curves with  marked points \cite[\S 10]{BH22+}, Bhargava and Ho prove that $100\%$ of elliptic curves over $\Q$ with an additional marked point have positive rank. In this note we provide an alternate proof which extends the result to global fields of characteristic not two or three.
\end{abstract}  

Let $\cM_{1,2}$ denote the moduli space of genus one curves with two marked points. Let $K$ be a global field of characteristic not equal to $2$ or $3$. Over $K$ the space $\cM_{1,2}$ can be identified with an open substack of the weighted projective stack $\cP(2,3,4)$ (see, e.g., \cite[\S 2]{Inc22}). 
We define a height function on $\cM_{1,2}$ by defining a height on $\cP(2,3,4)$. 

Let $\Val(K)$ denote the set of places of $K$, and let $\Val_0(K)$ (resp. $\Val_\infty(K)$) denote the set of finite places (resp. infinite places) of $K$. For each finite place $v\in \Val_0(K)$ let $\pi_v$ be a uniformizer.
For $x=(x_0,x_1,x_2)\in K^{3}- \{(0,0,0)\}$ define
\begin{equation}
 |x|_{{(2,3,4)},v}\defeq
 \begin{cases}
 \max\left\{|\pi_v|_v^{\lfloor v(x_0)/2\rfloor},|\pi_v|_v^{\lfloor v(x_1)/3\rfloor},|\pi_v|_v^{\lfloor v(x_2)/4\rfloor}\right\} & \text{ if } v\in \Val_0(K),\\
 \max\left\{|x_0|_v^{1/2},|x_1|_v^{1/3},|x_2|_v^{1/4}\right\} & \text{ if } v\in \Val_\infty(K).
 \end{cases}
 \end{equation}
 Then the \textit{height} of a point $x=[x_0:x_1:x_2]\in\cP(2,3,4)(K)$ is defined to be
\begin{equation}
\Ht_{(2,3,4)}(x)\defeq\prod_{v\in \Val(K)} |(x_0,x_1,x_2)|_{(2,3,4),v}.
\end{equation}

\begin{theorem}\label{thm:main}
    Let $K$ be a global field of characteristic not equal to $2$ or $3$.
    When ordered by height, $100\%$ of elliptic curves defined over $K$ with an additional marked point have positive rank. 
\end{theorem}

\begin{proof}
It follows from \cite[Theorem 4.1.1]{Phi24} and \cite[Theorem 8.3.2.2]{Dar21} (see also \cite[Theorem 4.05]{Phi25}) that for any open substack $X\subseteq \cP(2,3,4)$ we have the asymptotic
\begin{equation}\label{equation:equidistribution}
    \#\{x\in X(K) : \Ht_{(2,3,4)}(x)\leq B\}\sim \#\{x\in \cP(2,3,4)(K) : \Ht_{(2,3,4)}(x)\leq B\}.
\end{equation}
Let $\cY_1(N)$ be the modular curve parameterizing pairs $(E,P)$, where $E$ is an elliptic curve and $P$ is a point of order $N$. View $\cY_1(N)$ as a substack of $\cM_{1,2}$. 
Let $\cY_{tors}$ denote the union of all modular curves $\cY_1(N)$ with $\cY_1(N)(K)\neq \emptyset$. By uniform bounds for torsion of elliptic curves \cite{Lev68,Mer96}, $\cY_{tors}$ will be a finite union, and thus a closed substack of $\cM_{1,2}$. Applying the asymptotic (\ref{equation:equidistribution}) with $X=\cM_{1,2}$ and $X=\cM_{1,2}-\cY_{tors}$ proves the theorem.
\end{proof}

\bibliographystyle{amsalpha}
\bibliography{bibfile}

\end{document}